\documentclass[12pt]{article}
\usepackage[utf8]{inputenc}

\usepackage{amsmath}
\usepackage{mathtools}
\usepackage{amssymb}
\usepackage{amsthm}
\usepackage{graphicx}
\usepackage{amscd}
\usepackage{epic, eepic}
\usepackage{url}
\usepackage{color}
\usepackage[utf8]{inputenc} 
\usepackage{makebox}
\usepackage{comment}
\usepackage{enumerate}

\newcommand{\cP}{{\mathcal P}}

\newcommand{\cL}{{\mathcal L}}

\newcommand{\F}{{\mathbb F}}
\newcommand{\inti}{{\rm in}}

\newcommand{\AP}{A_{\mathcal{P}}}
\newcommand{\AL}{A_{\mathcal{L}}}

\newtheorem{theorem}{Theorem}[section]
\newtheorem{problem}[theorem]{Problem}
\newtheorem{lemma}[theorem]{Lemma}

\newtheorem{remark}[theorem]{Remark}
\newtheorem{constr}[theorem]{\bf Construction}

\newtheorem{prop}[theorem]{\bf Proposition}

\theoremstyle{definition} 
\newtheorem{defi}[theorem]{\bf Definition}

\DeclareMathOperator{\PG}{{PG}}

\title{Partitioning the projective plane to two incidence-rich parts }
\author{Zoltán Lóránt Nagy\thanks{ELTE Linear Hypergraphs  Research Group,
		E\"otv\"os Lor\'and University, Budapest, Hungary. Email: {zoltan.lorant.nagy@ttk.elte.hu}. The author is supported by the Hungarian Research Grant (NKFIH)  PD. 134953. and K. 124950}  }
\date{}

\begin{document}
	
	\maketitle
	
	\begin{abstract} An internal or friendly partition of a vertex set  $V(G)$ of a graph $G$ is a partition to two nonempty sets $A\cup B$ such that  every vertex has at least as many neighbours in its own class as in the other one.
		Motivated by Diwan's existence proof  on internal partitions of graphs with high girth, we give constructive proofs for the existence of internal partitions in the incidence graph of projective planes and discuss its geometric properties. In addition, we determine exactly the maximum possible difference between the  sizes of the neighbor set in  its own class and the neighbor set of the other class, that can be attained for all vertices at the same time for the incidence graphs of desarguesian planes of square order. \\
		
	\end{abstract}
	\section{Introduction}

	An {\em internal} or {\em friendly} partition of a graph is a partition of the vertices into two nonempty sets so that every vertex has at least as many neighbours in its own class as in the other one. Denote  the degree of a vertex $v$ of a graph $G$ by $d(v)$, and the number of neighbors of a vertices of $v$ belonging to a subset $A$
	of the vertex set $V(G)$ by $d_A(v)$. The notation $G|_{U}$ stands for the subgraph of $G$ induced by the vertex set $U$.  Then the condition on a partition $V(G) = A\cup B$ to be internal or friendly read as follows:  $d_A(v) \geq  d(v)/2, \  \forall v \in A$ and  $d_B(v) \geq d(v)/2, \ \forall v \in B$. 
	
	The problem of finding or showing the existence of internal partitions in graphs has a long history. The same concept was introduced by Gerber and Kobler \cite{gerber2004classes} under the name of {\em satisfactory partitions}, while Kristiansen, Hedetniemi and Hedetniemi considered a related problem on {\em graph alliances} \cite{kristiansen2002introduction}. The survey of Bazgan, Tuza and Vanderpooten \cite{bazgan2010satisfactory} describes early results on the area and discusses the complexity of the problem as well as how to find such partitions. 
	\\
	Stiebitz \cite{stiebitz1996decomposing} investigated a more general version where the degrees of the vertices might have different bounds depending on which part they belong to, and he proved that for every pair of functions
	$a, b : V \rightarrow \mathbb{N}^+$ satisfying $$ d_G(v) \geq a(v) +b(v) + 1, \ \forall v \in V,$$ there exists a partition of the vertex set $V(G) = A\cup B$, such that  $d_A(v) \geq  a(v), \  \forall v \in A$ and  $d_B(v) \geq b(v), \ \forall v \in B$.
	Kaneko proved  \cite{kaneko1998decomposition} that even the conditions  $d_A(v) \geq  a(v), \ \forall v \in A$ and $ d_B(v) \geq b(v),  \ \forall v \in B$ can be satisfied for arbitrary functions  $a, b : V \rightarrow \mathbb{N}^+$ such that $ d_G(v) \geq a(v) +b(v), \ \forall v \in V$, provided that $G$ is triangle-free. Note that the condition $ d_G(v) \geq a(v) +b(v), \ \forall v \in V$ cannot be assumed in general, since there are graphs which has no partition satisfying  $d_A(v) \geq  a(v), \forall v \in A$ and $ d_B(v) \geq b(v), \forall v \in B$.
	Likewise, there exist infinitely many graphs having no internal partitions, e.g., $K_{2n}$ and $K_{2n+1,2n+1}$. It is still a wide open problem, formulated by DeVos \cite{devos}, to determine whether for given $d$ there are only finitely many $d$-regular graphs without having an internal partitions. For the latest progress on the $d$-regular case, see \cite{ban, 5reg, linial, Zu}.  Several large classes of graphs have been shown to have internal partitions. Diwan proved  via an existence proof \cite{Diwan}   that if  a graph of girth at least $5$ has  minimum degree  at least $a+b-1$, then its vertex set has a suitable partition $A\cup B$ with minimum degrees $\delta_{G |_A}\geq a$ and $\delta_{G |_B}\geq b$, on the graph induced by $A$ and $B$, respectively.

	Diwan notes \cite{Diwan}  that it is not clear if the result above can be improved further for graphs with larger girth, since
	"\textit{the incidence graph of the projective plane of order 3 has 26 vertices,
		$\delta(G) = 4$, and girth six. It can be verified that this graph cannot be decomposed into two subgraphs with minimum degree 3.}"
	
	Our goal is two-wise in this paper. First, to  construct internal partitions in the incidence graphs of all (desarguesian) projective planes and hence to examine the behind geometric structure. 
	Second, to determine exactly how much the statement can be improved for the graph family in view, the incidence graphs of projective planes. In order to do so, we introduce the notion of \textit{$t$-internal partitions}.
	
	\begin{defi}[Intimacy of a graph]  A partition of the vertex set $V(G) = A\cup B$ of  a graph $G$ is called a \textit{ $t$-internal  or $t$-friendly partition} if  $d_A(v) \geq  d(v)/2+t, \  \forall v \in A$ and  $d_B(v) \geq d(v)/2+t, \ \forall v \in B$. The maximum integer value of $t$ for which there exists a $t$-internal partition in a graph is called the \textit{intimacy of the graph} and denoted by $\inti(G)$.
	\end{defi}
	
	Observe that the previously mentioned results of Stiebitz \cite{stiebitz1996decomposing} imply that $\inti(G)\geq -1$ for all regular graphs and  $\inti(G)\geq 0 \Longleftrightarrow$ $G$ has an internal partition.
	The intimacy is closely related to the Cheeger constant. Indeed, the isoperimetric number or Cheeger constant of a graph is $$\min_{ A\subset V, \ |A|\le |V|/2,} \frac{e(A,V\setminus A)}{|A|}.$$
	
	In other words, it is determined by a vertex subset $A$ of size at most $n/2$  which minimizes the average $\sum_{x\in A} \frac{d_B(x)}{|A|}=\sum_{x\in B} \frac{d_A(x)}{|B|}$. By contrast, for the intimacy, the value  ${d_B(x)}$ is required to be small for every $x \in A$ and the value  ${d_A(x)}$ is required to be small for every $x \in B$.

	Our main result is as follows.
	
	\begin{theorem}\label{mainsquare}
		The intimacy of the incidence graph $\mathcal{G}$ of a desarguesian projective plane of order $q$, $\PG(2,q)$,  is equal to $\lfloor\frac{\sqrt{q}-1}{2}\rfloor$ provided that $q$ is a square. 
	\end{theorem}
	Hence,  while the conditions concerning the internal partitions of the incidence graphs of $\Pi_q$ can not be improved in general, it actually can be for an infinite family, and the maximal possible improvement can be determined exactly.

	The paper is organised as follows. Section 2. is a preliminary section where we collect some useful well known eigenvalue techniques and results that  will be applied later on, such as the Expander Mixing Lemma. Section 3. is devoted to the proof of our main result, Theorem \ref{mainsquare}. In Section 4. we describe three different constructions for internal partitions in the incidence graph of projective planes of odd order. The simplest construction is geometric and works in every such plane. The second one is of algebraic flavour and exploits the coordinatisation. For the third one, the existence of a so-called \textit{ oval}, a $q+1$-arc is required in the plane. Finally, we show internal partitions in projective planes of even order $\PG(2,2^h)$ ($h>1$) relying on the existence of so-called \textit{maximal $(k; n)$-arc}s. Observe that when $q$ is even, the stricter conditions $d_A(v)>d_B(v) \ \forall v\in A$ and  $d_B(v)>d_A(v) \  \forall v\in B$ hold. Some  open problems and other results concerning the Zarankiewicz problem and the problem of large biholes are discussed in Section 5.
	
	\section{Preliminaries}

	We start this section  by recalling some basic facts and statements about projective planes as well as setting the notations.
	A \textit{projective plane} $\Pi$ is a point-line incidence structure such that any two points are on a unique
	line, and any two lines has a unique intersection point. It is said to be of order $q$ when all
	lines have $q+1$ points and all points are on $q+1$ lines. A projective plane  $\Pi_q$ of order
	$q$ has $q^2 +q+1$ points and as many lines.
	A desarguesian projective plane of order $q$, denoted by $\PG(2,q)$, is a plane over a finite Galois field $\F_q$,  such that the points  are the equivalence classes of the set $\F_q^3 \setminus {(0, 0, 0)}$ modulo the equivalence relation
	$\textbf{x} \equiv t\textbf{x}$, for all $t$ in $\F_q\setminus \{0\}$. Lines 
	consist of the equivalence classes of the non-zero points of $2$-dimensional subspaces of $ \F_q^3$, hence one can assign  as well the equivalence classes of the set $\F_q^3 \setminus {[0, 0, 0]}$ modulo the equivalence relation
	$\textbf{x} \equiv t\textbf{x}$, for all $t$ in $\F_q\setminus \{0\}$, such that 
	the line $[x:y:z]$ consists of those points $(a:b:c)$ for which $ax+by+cz=0$.
	The 
	triplets representing points and lines are called \textit{homogeneous coordinates} of the points and lines.

	Let $\mathcal{G}$ be a regular graph of valency $k$ on $n$ vertices with ordered spectrum $\lambda_1\ge \lambda_2\ge \ldots \ge\lambda_n$ and let $\lambda_2^*$ denote  $\lambda_2^*:=\max\{|\lambda_2|, |\lambda_n|\}$.
	
	The celebrated \textit{expander mixing lemma}
	due to Alon and Chung \cite{expmix} (see also \cite{Haemers}), says that there number of edges spanned by two subsets van be expressed as the expected number of edges according to the edge density
	between two subsets plus an error term, depending on $\lambda_2^*$.
	
	\begin{lemma}[Expander mixing lemma]\label{expandermixing}
		Let $S$ and $T$ be two subsets of the vertex set of a $d$-regular $\mathcal{G}$, of sizes $s$ and
		$t$, respectively. Let $e(S,T)$ be the number of edges $xy$ with $x \in  S$ and $y \in T$.
		Then we have 
		$$\mid e(S,T)-\frac{dst}{n}\mid \leq \lambda_2^*\sqrt{st(1-\frac{s}{n})(1-\frac{t}{n})}.$$
	\end{lemma}
	
	This statement has a variant as well for regular bipartite graphs. In this case, $\lambda_1=\lambda_2^*=|\lambda_n|$ would hold in the original version, however we can use $\lambda_2$ to replace $\lambda_2^*$ once one assumes that $S$ and $T$ are subsets of different partition classes of the bipartite graph. 
	
	\begin{lemma}[Expander mixing lemma, bipartite version]\label{expandermixingbip}
		Let $S\subset A$ and $T\subset B$ be two subsets of the vertex set of an $r$-regular bipartite graph $\mathcal{G}(A,B)$, of sizes $s$ and
		$t$, respectively. Let $e(S,T)$ be the number of edges $xy$ with $x \in  S$ and $y \in T$.
		Then we have 
		$$\mid e(S,T)-\frac{rst}{n}\mid \leq \lambda_2\sqrt{st(1-\frac{s}{n})(1-\frac{t}{n})}.$$
	\end{lemma}
	The bipartite version appeared first in the work of Haemers \cite{interlacing} in connection with the incidence structure of block designs. Here we used the form from  \cite{Lund}.

	We will use the following folklore statement  (c.f. \cite{Hoff}).
	
	\begin{lemma}[Eigenvalues of the point-line incidence matrix]\label{eigen} 
		Let $M$ denote the incidence matrix of a projective plane $\Pi_q$. 
		Then the eigenvalues of $M$ are $q+1$, and $\sqrt{q}$, as $MM^T=qI+J$ where $I$ stands for the unit matrix while $J$ denotes the all-one matrix.
	\end{lemma}

	There exists a partition of $\PG(n, q^h)$ into subgeometries isomorphic to $\PG(n, q)$
	if and only if $gcd(n+1, h) = 1$ (see e.g., Hirschfeld \cite{Hirsch}). In particular for $h = 2$,
	$\PG(n, q^2)$ can be partitioned into subgeometries $\PG(n, q)$ precisely when $n$ is even, and such subgeometries are called
	Baer subgeometries. 
	For the planar case $n=2$,  Bruck proved this Baer subplane decomposition result using the Singer cycle \cite{Bruck}.  
	
	\begin{lemma}\label{baer}
		If $g$ generates the (cyclic) Singer group  $\langle g \rangle$ 
		acting regularly on the points and lines of $\PG(2,q^2)$ then the point orbits of $\langle g^{q^2-q+1} \rangle$ form a partition of disjoint Baer subplanes, each isomorphic to $\PG(2,q)$.
	\end{lemma}

	\section{Maximum possible intimacy in projective planes}
	
	We start with the proof of Theorem \ref{mainsquare} which determines the intimacy for the incidence graph of $\PG(2,q)$ when $q$ is a square.
	\begin{proof}[Proof of Theorem \ref{mainsquare}]
		First we show that the intimacy can indeed attain $\lfloor(\sqrt{q}-1)/2\rfloor$.\\ Consider the decomposition of $\PG(2,q)$ to disjoint Baer subplanes $\Pi_{\sqrt{q}}^{(1)}, \ldots, \Pi_{\sqrt{q}}^{(q-\sqrt{q}+1)}.$  Each plane consist of $q+\sqrt{q}+1$ Baer sublines, which are incident with $\sqrt{q}+1$ points of their own subplanes, and every point in $\PG(2,q) \setminus \Pi_{\sqrt{q}}^{(j)}$ is incident with exactly one subline of $\Pi_{\sqrt{q}}^{(j)}$, for all $j$. Dually, 
		every line not in the subplane $\Pi_{\sqrt{q}}^{(j)}$  intersects the point set of $\Pi_{\sqrt{q}}^{(j)}$ in exactly one point. \\
		Consider the following construction. Let $\cP\cup \cL$ denote the union of point and line set of the plane $\PG(2,q)$ while $\cP_j\cup \cL_j$ stands for the union of point and line set of the Baer subplane $\Pi_{\sqrt{q}}^{(j)}$. Take the incidence graph $\mathcal{G}(\Pi_q)$ for $\Pi(q)=\PG(2,q)$.
		Let $A$ be defined as $$A:=\bigcup_{j=1}^{m} \cP_j\cup \cL_j$$ where $m:=\left\lfloor \frac{q-\sqrt{q}+1}{2}\right\rfloor$. Then the above reasoning on the intersection patterns of Bare sublines implies that ${\mathcal{G} |_A}$ is a regular graph of valency $\sqrt{q}+1+(m-1)$ and ${\mathcal{G} |_B}$ is a regular graph of valency $\sqrt{q}+1+m$ for $B=V(\mathcal{G})\setminus A$. Hence $\inti(\mathcal{G})\geq \frac{\sqrt{q}-1}{2}$ when $q$ is odd and $\inti(\mathcal{G})\geq \lfloor \frac{\sqrt{q}-1}{2}\rfloor$ when $q$ is even, taking into account that the valency of  $\mathcal{G}$ is $q+1$.\\
		
		Next, consider a pair of subsets $A$ and $B$ of $V(\mathcal{G})$ such that the partition $A\cup B$ is $m$-internal with $m= \inti(\mathcal{G})$. We may assume that  $|A\cap \cL|\leq \frac{q^2+q}{2}$ since $|A\cap \cL|+|B\cap \cL|=q^2+q+1$. We may also assume that $|A\cap \cL|\leq   |A\cap \cP|$. Let us introduce the notations $\AP:= A \cap \cP$ and $\AL:=A\cap \cL$. 
		By applying the Expander Mixing Lemma \ref{expandermixingbip} for $\AP$ and $\AL$,  we get that
		
		$$\mid e(\AP,\AL)-\frac{(q+1)|\AP||\AL|}{q^2+q+1}\mid \leq \lambda_2\sqrt{|\AP||\AL|(1-\frac{|\AP|}{q^2+q+1})(1-\frac{|\AL|}{q^2+q+1})}.$$
		
		Here we have $\lambda_2=\sqrt{q}$ in view of Lemma \ref{eigen}. 
		
		This in turn implies 
		\begin{equation}\label{eq:1}
			e(\AP,\AL)\leq \frac{(q+1)|\AP||\AL|}{q^2+q+1} + \sqrt{q|\AP||\AL|(1-\frac{|\AP|}{q^2+q+1})(1-\frac{|\AL|}{q^2+q+1})}.  
		\end{equation} 
		On the other hand, we have a lower bounds on $e(\AP,\AL)$  by the assumption on the intimacy, from which we obtain
		\begin{equation}\label{eq:0}
			e(\AP,\AL)\geq |\AP|\left(\frac{q+1}{2}+\inti(\mathcal{G})\right) 
		\end{equation}
		
		Suppose to the contrary that $\inti(\mathcal{G})\geq \sqrt{q}/2$.
		We distinguish between two cases.\\
		\textbf{Case 1.} $0<|\AP|<\frac{q^2+q+1}{2}$.\\
		Then the right hand side of Inequality  \ref{eq:1} can be bounded above by 
		
		\begin{equation}\label{eq:2}
			\frac{(q+1)|\AP|^2}{q^2+q+1} + \sqrt{q}|\AP|(1-\frac{|\AP|}{q^2+q+1}),  
		\end{equation} since it is a monotone increasing function in $|\AL|$ in the interval $[0, |\AP|]$. However, 
		
		\begin{equation}\label{eq:3}
			\frac{(q+1)|\AP|^2}{q^2+q+1} + \sqrt{q}|\AP|(1-\frac{|\AP|}{q^2+q+1})<   |\AP|\left(\frac{q+1}{2}+\frac{\sqrt{q}}{2}\right),
		\end{equation} which is a contradiction in view of Inequality \ref{eq:0}.
		Indeed, Inequality \ref{eq:3} is equivalent to
		\begin{equation}\label{eq:4}
			\frac{(q+1-\sqrt{q})|\AP|}{q^2+q+1} + \sqrt{q}<   \frac{q+\sqrt{q}+1}{2} \Longleftrightarrow |\AP|<\frac{q^2+q+1}{2},
		\end{equation} which holds according to our assumption.

		\textbf{Case 2.} $|\AP|>\frac{q^2+q+1}{2}$.\\
		
		Then the right hand side of Inequality  \ref{eq:1} can be bounded above by 
		
		\begin{equation}\label{eq:5}
			\frac{(q+1)|\AP|}{2} + \sqrt{q|\AP|\frac{q^2+q+1}{2}\left(\frac{q^2+q+1-|\AP|}{q^2+q+1}\right)\frac{1}{2}}.  
		\end{equation}  since it is a monotone increasing function in $|\AL|$ in the interval $[0,\frac{q^2+q+1}{2} ]$. However,

		\begin{equation}\label{eq:6}
			\sqrt{q|\AP|\frac{q^2+q+1}{2}\left(\frac{q^2+q+1-|\AP|}{q^2+q+1}\right)\frac{1}{2}}<\frac{\sqrt{q}}{2}|\AP|,   
		\end{equation} since

		\begin{equation}\label{eq:7}
			\sqrt{q\frac{1}{2}\left({q^2+q+1-|\AP|}\right)\frac{1}{2}}<\frac{\sqrt{q}}{2}\sqrt{|\AP|},   
		\end{equation} as we are in Case 2. This is again a contradiction with Inequality (\ref{eq:0}).
	\end{proof}
	
	\section{Internal partition in all desarguesian projective planes, geometric structure}
	
	\subsection{Odd order, several different constructions}
	We begin with a combinatorial construction which works for every projective plane of odd order.
	
	\begin{constr}\label{combin} Let $\Pi=\Pi_q$ be a projective plane of order $q$. Take the point set $\cP_1$ of $\frac{q+1}{2}$ concurrent lines through a fixed point $P$. Choose a line $\ell$ which is not incident to $P$.
		Take the line set $\cL_1$ of lines incident with those $\frac{q+1}{2}$ collinear points on $\ell$ which are in $\cP_1$. Note that the size of the intersection of $\ell$ and $\cP_1$ is precisely $\frac{q+1}{2}$.  Then $\cP_1\cup \cL_1$ and its complement $\cP_2\cup \cL_2$ form an internal partition in the incidence graph  of $\Pi_q$.
	\end{constr}
	
	\begin{proof} It is easy to see that every point is incident with $\frac{q+1}{2}$ lines of $\cL$, apart from the point in $\ell\cap \cP_1$. Also,  every line is incident with $\frac{q+1}{2}$ point of $\cP$, apart from the lines in $\cL_1$ which are incident with $P$. This result then in turn follows.
	\end{proof}
	
	\begin{remark} A slight variant of Construction \ref{combin} also satisfies the conditions, when we might drop $P$ from $\cP_1$  and we might drop  $\ell$ from $\cL_1$ as well, independently from the choice on $P$. 
	\end{remark}
	
	Observe that in the construction above, $\frac{q+1}{2}$ lines and equally many points have full degree $q+1$ within the internal partition classes $A=\cP_1\cup \cL_1$  and $B=\cP_2\cup \cL_2$ of the incidence graph.
	We show an algebraic construction which shows a different structure, provided that the plane is desarguesian. Let $S$ denote the set of square elements in $\F_q$, i.e., $S=\{a^2 \mid a\in \F_q^{\times} \}$.
	
	\begin{constr}\label{1mod4} 
		Let     $A=\cP_1\cup \cL_1$ be constructed as follows.\\ Let $U_\mathcal{P}:=\{(0:1:0), (1:0:0), (0:0:1) \}$ and  let $U_\mathcal{L}:=\{[0:1:0], [1:0:0], [0,:0:1] \}$
		$$\cP_1:=\{(x:y:1): y/x\in S \}\cup\{(0:y:1): y\in S \}\cup \{(x:0:1): x\not \in S \} \cup \{(x:1:0): x\in S \} \cup U_\mathcal{P} $$
		
		likewise,
		$$\cL_1:=\{[x:y:1]: y/x\in S \}\cup\{[0:y:1]: y\in S \}\cup \{[x:0:1]: x\not \in S \} \cup \{[x:1:0]: x\in S \} \cup U_\mathcal{L}. $$\qed
	\end{constr}
	
	If $q\equiv 3\pmod 4$, the  construction is slightly different. Note that in this case, $-1$ is not a square element in $\F_q$.
	\begin{constr}\label{3mod4}
		Let     $A=\cP_1\cup \cL_1$ be constructed as follows.\\Let $U_\mathcal{P}:=\{(0:1:0), (1:0:0), (0:0:1) \}$ and  let $U_\mathcal{L}:=\{[0:1:0], [1:0:0], [0,:0:1] \}$
		$$\cP_1:=\{(x:y:1): y/x\in S \}\cup\{(0:y:1): y\in \F_q^{\times} \}\cup  \{(x:1:0): x\not\in S \} \cup U_\mathcal{P} $$
		
		likewise,
		$$\cL_1:=\{[x:y:1]: y/x\in S \}\cup\{[0:y:1]:  y \in\F_q^{\times}  \} \cup \{[x:1:0]: x\not\in S \} \cup U_\mathcal{L}. $$\qed
	\end{constr}
	
	\begin{remark}
		Construction \ref{1mod4}  contains a point, namely $(0,0,1)$, for which $\frac{q-1}{2}$ complete lines are contained in  $\cP_1$, but not more. Construction \ref{3mod4}  contains a point, namely $(0,0,1)$, for which $\frac{q-1}{2}$ almost complete lines are contained in  $\cP_1$, i.e., complete lines in the affine part, without the ideal points. The dual statement holds for the lines. This structural distinction shows that these algebraic constructions are different from the combinatorial Construction \ref{combin}.  
	\end{remark}
	
	\begin{prop}
		Constructions \ref{1mod4} and \ref{3mod4}, together with their complements with respect to $\cP\cup \cL$ give an internal partition of the incidence graph $\mathcal{G}(\PG(2,q))$.  
	\end{prop}
	
	\begin{proof}
		The statement follows simply for the points of the lines of the set $U_\mathcal{L}$, taking into account that $-1$ is a square element and a non-square element, respectively, in cases $q\equiv 1 \pmod 4$ and $q\equiv -1 \pmod 4$. By duality it is enough to check that each point of $\cP_1$ is incident with at least  $\frac{q+1}{2}$ lines of $\cL_1$ and   each point of $\cP\setminus\cP_1$ is incident with at least  $\frac{q+1}{2}$ lines of $\cL\setminus \cL_1$. If one of the homogeneous coordinates of a point $P(x:y:z)$ is $0$, this can be verified easily. Suppose for instance that $z=0$, $y\neq 0$, i.e., $P=(x:1:0)$. Then the number of incidences with a line $[a,b,c]$ transforms to the number of solutions  of $\{ax+b=0, c=1 \mbox{ \ or \ } ax+b=0, c=0 \}\Longleftrightarrow \{ \frac{a}{b}x=-1, c=1 \mbox{ \ or \ } \frac{a}{1}x=-1, c=0\mbox{ \ or \ } b=a=0, c=1\}$, where have  information on whether $a,b\in S$ holds or not holds. The general case when $xyz\neq 0$ is similar, when  the number of incidences with a line $[a:b:c]$ transforms to the number of solutions  of $ax+by+cz=0 $, that is, 
		$\{ ax+by=-1, c=1 \mbox{ \ or \ } ax+y=0, c=0  \mbox{ \ or \ } ax+z=0, b=0 \mbox{ \ or \ }  by+z=0, a=0\}$ for homogeneous triples $[a,b,c]$, where have on information on whether $a,b\in S$ holds or not holds. Hence simple case analysis completes the proof.
	\end{proof}
	
	\begin{remark} It is easy to check that erasing    $U_\mathcal{P}\cup U_\mathcal{L}$ from either Construction \ref{1mod4} or \ref{3mod4} results a good construction as well.
	\end{remark}
	
	For the next construction, we introduce the arcs in projective planes. For more details on these objects, we refer to the excellent books of Hirschfeld \cite{Hirsch}, and of Kiss and Szőnyi \cite{KSz}. We use the notations of \cite{KSz}.
	
	\begin{defi}[Arcs] A
		set of points in a plane is called an \textit{arc} if no three of its
		points are collinear. An arc is called a $k$-arc, if it contains exactly $k$ points. \\     If the order of the plane is $q$,  then a  $q+1$-arc is called an \textit{oval}, and 
		a $(q+2)$-arc is called a \textit{hyperoval}.
	\end{defi}

	\begin{defi} Let $K$ be a point set and $\ell$ be a line in a projective plane. Then  $\ell$ is called \vspace{-0.4cm}
		\begin{itemize}
			\item  a \textit{$t$-secant} to $K$ if $|K \cap \ell| = t$ (for $t>1$);\vspace{-0.2cm}
			\item a \textit{tangent} to $K$ if $|K \cap \ell| = 1$;\vspace{-0.2cm}
			\item a \textit{skew line} to $K$ if $|K \cap \ell| =0$.
		\end{itemize} 
	\end{defi}\vspace{-0.2cm}
	Note that by definition, there exists a unique tangent on every point of an oval. Simple double counting shows that if $q$ is odd, then once a point is not on a given $q+1$-arc (oval) $K$, and  lies on a tangent, then it lies on exactly two tangents to $K$.
	We also mention that due to the well-known theorem of Segre, if $q$ is odd then  each oval is a conic in a desarguesian plane $\PG(2,q)$.

	\begin{defi}
		Let $q$ be odd and let $\mathcal{O}$ be an oval in the projective plane $\Pi_q$.  A point $\mathcal{O}$ is called \textit{interior or exterior}  point with respect to $\mathcal{O}$ according
		to whether it lies on 0 or 2 tangents to $\mathcal{O}$.
	\end{defi}
	
	We recall a simple lemma from \cite[Lemma 6.14]{KSz}.
	\begin{lemma}\label{borrowed}
		Let $q$ be odd and $\mathcal{O}$ be an oval in a projective plane $\Pi_q$. Then there are exactly  $\frac{(q+1)q}{2}$
		exterior points and   $\frac{(q-1)q}{2}$ interior points with respect to $\mathcal{O}$.
		If the line $\ell$ is not a tangent to $\mathcal{O}$, then exactly half of the points of $\ell\setminus\mathcal{O}$
		are exterior and half of them are interior points with respect to $\mathcal{O}$.
	\end{lemma}
	
	\begin{constr}\label{oval}
		Let  $\cP_1\cup \cL_1$ be defined as follows. 
		
		Let $\cL_1:=\{\mbox{skew lines to $\mathcal{O}$}\}$ and $\cP_1:=\{\mbox{interior points of $\mathcal{O}$}\}$.
		
		Let $\cL_1^*:=\{\mbox{skew lines and tangents to $\mathcal{O}$}\}$ and $\cP_1^*:=\{\mbox{exterior points of $\mathcal{O}$}\}$.
	\end{constr}
	\begin{prop}
		Construction \ref{oval} provides two different internal partitions, both $\cP_1\cup~\cL_1$ and $\cP_1^*\cup \cL_1^*$ together with their complements with respect to the plane form an example.   
	\end{prop}
	
	\begin{proof}
		It is straightforward that $\cP_1\cup\cL_1$ induces a regular subgraph of valency $\frac{q+1}{2}$, and in addition, its complement forms  an induced subgraph of minimum degree $\frac{q+1}{2}$ in view of \ref{borrowed}. Similarly, each skew line and tangent is incident with at least  $\frac{q+1}{2}$ exterior point while every exterior point is incident with $2$ tangents and $\frac{q-1}{2}$  skew lines. Likewise, $2$-secants are incident with $\frac{q+3}{2}$ interior points or points of $\mathcal{O}$ while interior points are incident with $\frac{q+1}{2}$ $2$-secants via Lemma \ref{borrowed} and the points of  $\mathcal{O}$ are incident with $q$ $2$-secants. This in turn proves the statement. 
	\end{proof}

	\subsection{Even order case, relation to maximal $(k; n)$-arcs}
	
	In a finite projective plane $\Pi_q$ of order $q$, any set $K$ of $k$ points is called a  \textit{$(k; n)$-arc} if $\max \{|\ell \cap K| : \ell\in \mathcal{L}(\Pi_q) \} =n$. For given $q$ and $n$, the cardinality of the point set of the $(k; n)$-arc $K$ can never exceed $(n- 1)(q + 1) + 1$, and
	a $(k; n)$-arc with that number of points will be called  a \textit{maximal $(k; n)$-arc.}
	Equivalently, a maximal $(k; n)$-arc can be defined as a nonempty  point set of the projective plane which meets
	every line in either $n$ points or in zero. While in $\PG(2,q)$ with $q$ odd, no non-trivial maximal arcs exist \cite{BallBlokhuis}, we do have examples in the $q$ even case. From a classical result of Denniston \cite{Denniston}, we know that maximal  {$(k; n)$-arcs} exist in $\PG(2,q)$, $q=2^h$, for every divisor $n$ of $q$.

	\begin{constr}\label{evencase}
		Let $\mathcal{M}$ be a maximal $(\frac{q^2}{2}-\frac{q}{2}; \frac{q}{2})$-arc  of a projective plane $\PG(2,q)$ and $\ell$ one of its $\frac{q}{2}$-secant lines. Let $\cP_1$ be the point set obtained as the symmetric difference of these objects $\mathcal{M}\triangle\ell$. Let $\cL_1$ be the set of lines through the points of $\ell \setminus \mathcal{M}$ which intersect $\mathcal{M}$. 
	\end{constr}
	\begin{prop} Using the notions of Construction \ref{evencase}, 
		$\cP_1\cup \cL_1$ and its complement $\cP_2\cup \cL_2$ form an internal partition in the incidence graph  of $\PG(2,q)$.
	\end{prop}
	\begin{proof}
		First observe that basic double counting yields that every point $P\not \in \mathcal{M}$ is on exactly $q-1$ $q/2$-secants of $\mathcal{M}$ and $2$ skew lines to $\mathcal{M}$. Moreover, the cardinality of skew lines to $\mathcal{M}$ is exactly $q+2$. (These form a hyperoval in the dual plane as observed by Cossu.)
		This means that in the construction  $\cP_1\cup \cL_1$, the points $\ell \setminus \mathcal{M}$ will be incident with $q-1$ $(q/2+1)$-secants while the points of $\mathcal{M} \setminus \ell$ are incident with only $(q/2-1)$-secant lines  apart from $\ell$. On the other hand, the points of $\mathcal{M}\setminus \ell$ are incident with exactly $q/2+1$ $(q/2+1)$-secants (which are exactly the lines going through the points of $\ell \setminus \mathcal{M}$). For every point $Q$ outside $\ell \cup \mathcal{M}$, also at most $q/2+1$ $(q/2+1)$-secants would be created, determined by the lines going through the points of $\ell \setminus \mathcal{M}$, but two of them are tangent lines, obtained from skew lines to $\mathcal{M}$. This shows that $\deg_{\cP_i\cup \cL_i}(v)\geq q/2+1$ for all $v\in \cP_i$ ($i\in \{1,2\}$).\\
		It is even simpler to show $\deg_{\cP_i\cup \cL_i}(v)\geq q/2+1$ for all $v\in \cL_i$. Indeed, $\deg_{\cP_1\cup \cL_1}(v)= q/2+1$ for all $v\in \cL_1$ by the definition as we obtained $q/2+1$-secants from non-skew lines by adding  extra point in $\ell$. Lines skew to $\mathcal{M}$ became $1$-secants in the construction, thus $\deg_{\cP_2\cup \cL_2}(v)= q$ for $q+2$ vertices $v\in \cL_2$ in the incidence graph. Finally, lines through points of  $\mathcal{M}\cap \ell$ 
		became $q/2-1$-secants apart from $\ell$ itself, so for these lines, we have  $\deg_{\cP_2\cup \cL_2}(v)= q/2+2$.
	\end{proof}
	
	\section{Final remarks and open problems}
	
	In the previous sections, we showed several constructions with different structures that provided internal partitions for  the incidence graph of projective planes. If the prime power order has even exponent, much stronger conditions also hold for a suitable partition, which was grabbed by the concept of intimacy. It would be interesting to see whether the observation of Diwan for the order $q=3$ case, $\inti(\mathcal{G}(PG(2,3))=0$, relied essentially on the fact that the order was small, or similar statements can be made for larger prime numbers as well.
	
	\begin{problem}  Is it true that  the intimacy of the incidence graph of $\PG(2,p)$ is always positive for  large enough prime numbers $p$? 
	\end{problem}
	
	Note that results concerning the intimacy or in general, concerning incidence-rich substructures are closely related to exact results of certain Zarankiewicz numbers, see \cite{Dam, Dz}. Indeed, it is very natural to seek extremal $C_4$ free subgraphs in incidence-rich spanned subgraphs of the incidence graph of a projective plane. Incidence graphs of block designs with large biholes (incidence free sets induced by two partite classes of the bipartite graph) have been investigated recently by Adriaensen, Mattheus and Spiro \cite{Adria}, which is also a related subject. Indeed, if $\cP_1\cup \cL_1$ is incidence free, then the union of the point set and line set $(\cP\setminus \cP_1)\cup \cL_1$ must be rich in   incidences. These connections motivate the problems below.
	\begin{problem}
		Describe in general internal partitions of incidence graphs of block designs or  incidence structures.   
	\end{problem}
	
	\begin{problem}
		Determine the intimacy of the incidence graph of other block designs.
	\end{problem}
	\textbf{Acknowledgement.}
	I would like to thank Tamás Héger for his suggestion concerning Construction \ref{oval}.

\end{document}